\documentclass[12pt]{amsart}
\usepackage[T1]{fontenc}
\usepackage{amsmath,amssymb,verbatim,graphicx,hyperref}
\usepackage[cmtip,all]{xy}
\usepackage{amsfonts}
\usepackage{times}
\usepackage{amscd}
\usepackage{epsfig}
\usepackage{float}
\usepackage{amsthm}
\usepackage{latexsym}
\usepackage{imakeidx}
\usepackage[active]{srcltx}

\newdimen\plxleft
\newdimen\plxdown
\newdimen\plxright
\newdimen\plxtemp

\def\sob#1#2#3#4#5{
  \setbox0\hbox{#1}\setbox1\hbox{$_\mathchar'454$}\setbox2\hbox{p}%
  \plxright=#2\wd0 \advance\plxright by-#3\wd1
  \plxdown=#5\ht1 \advance\plxdown by-#4\ht0
  \plxleft=\plxright \advance\plxleft by\wd1
  \plxtemp=-\plxdown \advance\plxtemp by\dp2 \dp1=\plxtemp
  \leavevmode
  \kern\plxright\lower\plxdown\box1\kern-\plxleft #1}

\def\e{\sob e{.50}{.35}{0}{.93}}

\newtheorem{thm}{Theorem}[section]
\newtheorem{lem}[thm]{Lemma}

\newtheorem{prop}[thm]{Proposition}

\newtheorem*{thmA}{Theorem A}
\newtheorem*{thmB}{Theorem B}

\newtheorem*{thmM}{Main Theorem}

\theoremstyle{definition}
\newtheorem{dfn}[thm]{Definition}

\theoremstyle{remark}

\newtheorem*{StatThmA}{Static Theorem}

\numberwithin{equation}{section}

\def\mod{\mathrm{mod}}

\newcommand{\ol}{\overline}

\newcommand{\diam}{\mathrm{diam}}

\newcommand{\ga}{\gamma}
\newcommand{\Ga}{\Gamma}
\newcommand{\da}{\delta}

\newcommand{\al}{\alpha}
\newcommand{\be}{\beta}
\newcommand{\eps}{\varepsilon}

\newcommand{\sm}{\setminus}

\newcommand{\pc}{\mathcal{PC}}

      \def\C{\mathbb{C}}

\def\P{\mathbb{P}}            \def\R{\mathbb{R}}
      
\def\Z{\mathbb{Z}}

        \def\Bc{\mathcal{B}}        
                \def\Fc{\mathcal{F}}

\newcommand{\Kc}{\mathcal{K}}

\def\Pc{\mathcal{P}}

\renewcommand\le{\leqslant}
\renewcommand\ge{\geqslant}
\def\0{\varnothing}
\def\d{\partial}

\begin{document}

\date{\today}

\title[No a priori bounds]
{Maps with no a priori bounds}

\dedicatory{Dedicated to the memory of O. M. Sharkovsky}

\begin{abstract}
The modulus of a polynomial-like (PL) map is an important invariant
that controls distortion of the straightening map and, hence, geometry of the corresponding PL Julia set.
Lower bounds on the modulus, called \emph{complex a priori bounds}, are known in a great variety of contexts.
For any rational function we complement this by an upper bound for moduli of PL maps in the satellite case  that depends
\emph{only on the relative period and the degree of the PL map.} This rules out a priori bounds in the satellite case
with unbounded relative periods. We also apply our tools to obtain lower bounds for hyperbolic lengths of
geodesics in the infinitely renormalizable case, and to show that moduli of annuli must converge to $0$ for a sequence
of arbitrary renormalizations, under several conditions all of which are shown to be necessary.
\end{abstract}

\author[A.~Blokh]{Alexander~Blokh}

\thanks{The second named author acknowledges support of ISF grant 1226/17}

\thanks{The third named author was partially
supported by NSF grant DMS--1807558}

\author[G.~Levin]{Genadi~Levin}


\author[L.~Oversteegen]{Lex Oversteegen}

\author[V.~Timorin]{Vladlen~Timorin}

\thanks{The fourth named author was supported by
the HSE University Basic Research Program.
}

\address[Alexander~Blokh and Lex~Oversteegen]
{Department of Mathematics\\ University of Alabama at Birmingham\\
Birmingham, AL 35294}
\address[Genadi~Levin]
{Institute of Mathematics\\
The Hebrew University of Je\-ru\-salem\\
Givat Ram 91904, Jerusalem, Israel}
\address[Vladlen~Timorin]
{Faculty of Mathematics\\
HSE University\\
6 Usacheva str., Mo\-scow, Russia, 119048}

\email[Alexander~Blokh]{ablokh@uab.edu}
\email[Genadi~Levin]{genady.levin@mail.huji.ac.il}
\email[Lex~Oversteegen]{overstee@uab.edu}
\email[Vladlen~Timorin]{vtimorin@hse.ru}



\maketitle

\section{Introduction}\label{sec1}

Since Sullivan's work on Feigenbaum's universality \cite{Su}, complex a priori bo\-unds for
polynomial-like renormalizations play a key role in polynomial dynamics.
Such bounds, often hard to prove, are established and crucially used for various classes of
polynomials (see, e.g., \cite{Su,LS,GSb,LyYab,GS,Ly,kahn,kahnl1,kahnl2,shen,KSvS}).
They imply that the Julia set $J(f)$ is locally connected and, in a lot of cases,
rigidity of the corresponding maps in the considered family (if the latter is the quadratic family, this means the MLC conjecture at the corresponding parameters). Thus, bounds are an important tool in complex dynamics.

On the other hand, counterexamples by  Douady and Hubbard of satellite infinitely
renormalizable quadratic polynomials
with non-locally connected Julia sets show that the bounds do not always hold (see \cite{M,So}
for qualitative versions and \cite{L,L1,L1add} and \cite{CS} for quantitative ones).
These and similar examples of satellite renormalizations with ``no bounds''
are based on the polynomial-like (PL) connected Julia sets keeping definite size, deal with \emph{polynomials},
and require a \emph{sequence} of consecutive satellite renormalizations
with fast convergence to zero of the rotation
numbers at the $\alpha$-fixed points of renormalizations
(\cite{L1,L1add}, see also \cite{CS}). In particular, in all those examples
the relative periods of the renormalizations (i.e., the denominators of the rotation numbers) converge fast to infinity.
The authors are not aware of other cases where the ``no bounds'' condition
have been established.
So it has been unclear how general the phenomenon is and no explicit upper bounds have been known.

Motivated by the desire to clarify the mechanism behind
the ``no bound'' phenomenon, we study renormalizations of any rational function.
Given a PL map $P:U\to V$ \cite{DH} with connected filled Julia set $K^*$ call $U\setminus K^*$ a \emph{root
annulus (of $K^*$)}; the degree of $P$ is always denoted by $d^*$. If in a satellite PL
cycle $\Kc$ of connected PL Julia sets any connected component $Y$ of the union $\Kc^+$ of all
sets from $\Kc$ consists of $s$ PL Julia sets, then $s$ is called the \emph{relative period} (of $\Kc$), the (unique) common point $\al$
of these $s$ PL Julia sets is called a \emph{base point (of $\Kc$)}, and the orbit $\Bc$ of $\al$ is called the \emph{base cycle} of $\Kc$.
See Definitions \ref{d:renorm}-\ref{d:satel1} for more details.
Main Theorem is proven in Section \ref{ss:dynamic}.

\medskip

\begin{thmM}
Let $f$ be a rational function of degree $d\ge 2$ with a satellite PL cycle $\Kc$ of relative period $s\ge 2$.
Then, for any root annulus $A$ of any filled PL Julia set $K^*\in \Kc$,
$$
\mod(A)<\frac{d^*\pi}{\ln(4(s+1))}\le \frac{2^{2d-2}\pi}{\ln(4(s+1))}
$$
and if $f$ is a polynomial, then
$$
\mod(A)<\frac{d^*\pi}{\ln(4(s+1))}\le \frac{2^{d-1}\pi}{ \ln(4(s+1))}
$$
\end{thmM}

\medskip

For example, consider the family of infinitely renormalizable polynomials $f_c(z)=z^2+c$
such that all renormalizations are satellite. To each such $f_c$
one associates a sequence of relative rotation numbers of consecutive renormalizations. Then, by our results,
for a generic subset of the space of all such sequences the associated polynomials admit no a priori bounds.

Another application of our tools deals with geometry of the postcritical set in the
infinitely renormalizable case.

For a compactum $Y\subset \P^1$ such that $X=\P^1\sm Y$ is hyperbolic, let $\ell_X(\ga)$ be the hyperbolic length of a Jordan curve
$\ga\subset X$ in $X$. Given a geodesic $\gamma$ in a hyperbolic Riemann surface $X$, say that an annulus in $X$ is \emph{homotopic} to
$\ga$ if a simple closed curve in $A$ that is a deformation retract of $A$ is homotopic to $\gamma$ in $X$.
Bounded geometry of a postcritical infinitely renormalizable set $Y$ (with respect to a given sequence of renormalizations) means that for some $L>1$,
each level of renormalization, each PL Julia set $K^*$ of this level, and any simple closed geodesic in $X=\P^1\setminus Y$ homotopic in $X$
to a fundamental annulus of $K^*$, we have $\ell_X(\ga)\in [1/L, L]$.

Recently, D. Cheraghi [virtual workshop ``Many faces of renormalization'',
Simons Center for Geometry and Physics, 2021] and M. Pedramfar
[Workshop ``On geometric complexity of Julia sets'', 2018, B\e dlewo, Poland]
announced their joint results on the geometry and topology of a postcritical infinitely renormalizable set $\pc_f$
for a quadratic polynomial $f(z)$
where all renormalizations are satellite. Then, a \emph{(combinatorial) rational rotation number} $k/s\in (-1/2, 1/2]$ is associated with
a PL cycle, describing how its PL Julia sets rotate about their base points.
The results of D. Cheraghi and M. Pedramfar concern the cases when
all such rotation numbers are of \emph{high type} (see \cite{IS})
and state, in particular, that the postcritical set $\pc_f$ of such a polynomial $f$ has bounded geometry
if and only if the set of rotation numbers is bounded away from $0$.

Our results complement these. However, unlike the definition of bounded geometry,
we require that a simple geodesic $\ga$ goes not only around
the chunk of a postcritical infinitely renormalizable set $\pc_f$ contained in a PL Julia set $K^*$
but also around $\{\al\}=\Bc\cap K^*$ (recall that $\Bc$ is the corresponding base cycle of $\Kc$).
Corollary A follows from the more general, but also more technical,  Theorem \ref{t:hyper}.

\begin{thmA}
Consider a polynomial $f$ with a satellite PL cycle $\Kc$ of relative period $s>1$ that takes part in the formation of
an infinitely-renormalizable set $S$ (see Def. \ref{d:inf-renorm1}). Suppose that all (finite) critical points of $f$ belong to $S$, and set $X=\P^1\sm (\Bc\cup \pc_f)$.
Let $\ga$ be a simple closed geodesic in $X$ homotopic in $X$ to a fundamental annulus of a PL Julia set $K^*\in \Kc^+$.
Then $\ell_X(\ga)$ is bounded below by a universal function of $s$ and $d^*$ of order $O(\ln\,\ln s)$ as $s\to\infty$.
\end{thmA}

\medskip

Thus, if all critical points of a rational function $f$ belong to an infinitely renormalizable set then an adding
the base cycle to the postcritical set makes the situation with the length of simple closed geodesics universally
opposite to that described by Cheraghi and Pedramfar. Notice that in our setting the assumptions on $f$ are less restrictive.
Corollary A is proven in Section \ref{s:infrenorm}.

The Main Theorem shows that for rational functions $f_i$ of degree $d$
with satellite PL cycles of relative periods $s_i\to\infty$, the moduli of root annuli tend (uniformly in  $s_i$) to zero.
As we show in Theorem B in Section \ref{s:b}, this conclusion holds not only in the satellite case.
Moreover, no conditions in Theorem B can be dropped (see Proposition \ref{ex:Ex} and a short discussion before it).

\medskip

\begin{thmB}
Let $\{f_n\}$ be a sequence of degree $d\ge 2$ rational functions that converges
to a rational function $f$ of degree $d$. Assume that for each $n$ there is a renormalization
$f_n^{q_n}: U_n\to V_n$ of period $q_n$ with connected PL Julia set $K_n$, and $K_n\to K$ in the Hausdorff metric.
Suppose that the following holds:

\begin{enumerate}

\item there are at least two points in $K$,

\item the sequence $q_n$ tends to infinity,

\item no parabolic periodic domain of $f$ contains
$K$ (e.g., if $f$ has no parabolic points).

\end{enumerate}

Then $\mod(U_n\sm K_n)\to 0$.

\end{thmB}

\medskip

Let us comment briefly on our methods. The proof of the Main Theorem uses the solution of Teichmueller's extremal problem and is otherwise elementary.
This allows us to prove a completely general and simple (depending only on the degree and the relative period) upper bound -- for a single satellite renormalization.

Corollary A follows from more technical Lemma \ref{l:pc1} and Theorem \ref{t:hyper} that
extend some estimates used in the proof of the Main Theorem and rely upon
results of \cite{Mas85} relating hyperbolic lengths of closed geodesics and moduli of the associated annuli.

Finally, unlike the Main Theorem, Theorem B deals with \emph{arbitrary renormalizations}; thus, it
contains restrictive conditions that guarantee convergence to zero of moduli of the corresponding annuli.
The proof of the fact that these conditions are necessary (see Proposition \ref{ex:Ex}) is based on the existence of universal complex bounds for
real unimodal maps \cite{LS}, \cite{GSb}, \cite{LyYab} and the ``parabolic implosion'' technique
involving Lavaurs' map (see, e.g., \cite{Z}).

\subsubsection*{Acknowledgements}
The second named author thanks Davoud Cheraghi for an argument that simplified a proof
of a very preliminary qualitative version of the Main Theorem with the base period $1$ 
for quadratic polynomials
on an earlier stage of the work.
He also acknowledges Institute of Mathematics of PAN (Warsaw) for its hospitality during his sabbatical leave in Spring 2022 when this work was mostly done.
The fourth named author thanks MSRI for stimulating and comfortable (virtual) environment in Spring 2022,
 which gave distraction from the horrors of real life.

 \section{Proof of the Main Theorem}\label{s:main}

Let $\P^1$ be the Riemann sphere. For a compact subset $X\subset\C$, let $\diam_e(X)$
be its Euclidian diameter. For a finite set $\Pc\subset \C$, let $\da(\Pc)$ be
the shortest Euclidian distance between two distinct points of $\Pc$.
Write $C_f$ for the set of critical points of the map $f$.

\subsection{The Teichmueller extremal problem}\label{ss:extremal}

Section \ref{ss:extremal} describes classical geometric inequalities arising from extremal problems
of conformal geometry (see, e.g., \cite{ahl}, Sections 4.11 -- 4.12, for details).
Recall the \emph{Teichmueller extremal problem}: find the maximal
value of $\mod(U\sm Z)$, where an open topological disk $U\subset\C$
and a full continuum $Z$ are such that $0$, $-1\in Z$ while
$U$ does not contain the disk $\{z\in\C\mid |z|\le \eps\}$.
In other words, the distance between the complement of $U$ and $0$ is less than or equal to $\eps$.
It reads that the maximal $m=\mod(U\sm Z)$ is attained when $Z=[-1,0]$ while $U=\C\sm [\eps,\infty)$.
Write $\tau^{-1}(\eps)$ for the modulus of $U\sm Z$.
Here $\tau^{-1}$ is some increasing function defined on $\R_{>0}$,
 this function can be expressed in terms of elliptic integrals.
Let $\tau$ be the inverse function of $\tau^{-1}$; then
$$
\tau(m)>
16e^{-\pi/2m}.
$$

Let us now restate the Teichmueller extremal problem. For any $U$ and $Z$
from its statement, $m=\mod(U\sm Z)$ satisfies the inequality $\tau(m)\le \eps$. If
$m=\mod(U\sm Z)$ is fixed, then, clearly, the disk $\{z\in\C\mid |z|<\tau(m)\}$ is contained in $U$.

\medskip

\begin{lem}\label{l:euclid}
If $Z\subset \C$ is a full continuum and $V\supset Z$ is an open Jordan disk, then
the minimal distance between a point $x\in Z$ and
a point $y\in \C\sm V$ is greater than $\frac{\tau(\mod(V\sm Z))\cdot \diam_e(Z)}2$.
\end{lem}

\medskip

\begin{proof} Note that if $Y$ is a compactum
of diameter $D$ in a metric space $(X, d)$ then
$d(y, z)\ge D/2$ for any point $y$ and a point $z\in Y$ which is the farthest from $y$
point of $Y$. Indeed, take $a, b\in Y$ with $d(a, b)=D$; then $d(y, a)\ge D/2$ or $d(y, b)\ge D/2$.
This general claim is used below.

To prove the lemma let $x=0$. Moreover, by the above we can scale $Z$ by the factor $\ga\le 2/\diam_e(Z)$ so that
the point $z\in Z$ of maximal $|z|$ maps to $-1$. Choose $y\in \C\sm V$ so that $|y|$ is minimal
over $\C\sm V$. By the restatement of the Teichm\"uller extremal problem, $|y|\ge \tau(\mod(V\sm Z))$. Scaling everything
back, we see that $|y|/\ga \ge \tau(\mod(V\sm Z))/\ga \ge \frac{\tau(\mod(V\sm Z))\cdot \diam_e(Z)}2$.
\end{proof}

Lemma \ref{l:pack} estimates $\delta(\Pc)$ for finite sets $\Pc\subset \C$.

\medskip

\begin{lem} \label{l:pack}
Let $\Pc\subset \C$ be a set of $t>1$ points. Then $\delta(\Pc)< \frac{2}{\sqrt{t}}\diam_e(\Pc)$.
\end{lem}

\medskip

\begin{proof}
Since $1<\frac{2}{\sqrt{t}}$ for $t\le 3$ we may assume that $t\ge 4$. We may also
assume that $\diam_e(\Pc)=1$. Since
the disks of radius $\delta/2$ centered at all points of $\Pc$ are pairwise disjoint
and contained in a convex set of diameter $1+\delta$,
the total area $s\pi\delta^2/4$ of these disks is less than $\pi(1+\delta)^2/4$,
the maximal area a planar set of diameter $1+\delta$ can have
(see p. 239, ex. 610a of \cite{YaB}). Thus, $\frac{t\delta^2}{4}<\frac{(1+\delta)^2}{4}$ which
implies that $\delta<\frac{1+\sqrt{t}}{t-1}\le \frac{2}{\sqrt{t}}$
where the last inequality is based on the fact that $t\ge 4$.
\end{proof}

The next theorem is instrumental for the proof of the Main Theorem.

\medskip

\begin{StatThmA}
Let $\Pc=\{\al, z_1, \dots, z_t\}\subset \P^1$, $t\ge 3$. Then there are distinct points
$z_r, z_l$ such that for any annulus $A\subset\P^1$ separating $\{\al,z_i\}$ from $\Pc\sm\{\al, z_i\}$
($i=r, l$), we have $\mod(A)<\frac{\pi}{\ln(4t)}.$ Thus,
for any choice of $w\in\Pc$, there exists $z_k\in\Pc, z_k\ne w$ such that,
for any annulus $A\subset\P^1$ separating $\{\al,z_k\}$ from $\Pc\sm\{\al, z_k\}$, we have
$\mod(A)<\frac{\pi}{\ln(4t)}.$
\end{StatThmA}

\begin{proof}
Using a suitable Moebius transformation and the conformal invariance of the modulus we assume that
$\al=\infty$ and show that there are two points $z_r\ne z_l\in\Pc$ such that for $z'=z_r$ or $z'=z_l$
and for every annulus $A$ in $\C$ separating $\{z',\infty\}$ from $\Pc\sm\{z'\}$ we have $\mod(A)<\frac{\pi}{\ln(4t)}$.
We may assume that $\diam_e(\Pc\sm \infty)=1$. Let $z_r, z_l\in\Pc$ such that $|z_r-z_l|=\da(\Pc\sm \{\infty\})$ and show
that they are the desired points (it suffices to prove it for $z_r$).
Consider an annulus $A$ separating $\{z_r,\infty\}$ from $\Pc\sm\{z_r, \infty\}$.
Then $z_l\in \Pc\sm\{\infty, z_r\}\subset Z$ where $Z$ is
the bounded complementary component of $A$. Hence $\diam_e(Z)\ge \diam_e(\Pc\sm\{\infty, z_r\})\ge \frac12$
as otherwise $|z_r-z_l|<\frac12$ by the choice of $z_r, z_l$ and the Euclidian distance between any two points of $\Pc\sm \infty$
is less than $1=\diam_e(\Pc\sm \infty)$ which is absurd.
By Lemma \ref{l:euclid}, 
$$
\frac{\tau(\mod(A))\cdot 1/2}{2}\le \frac{\tau(\mod(A))\cdot \diam_e(Z)}{2}\le |z_r-z_l|\le \frac{2}{\sqrt{t}}
$$
and so $\tau(\mod(A))\le \frac{8}{\sqrt{t}}$; since $16e^{-\pi/2\mod(A)}<\tau(\mod(A))$, then 
$\mod(A)<\frac{\pi}{\ln(4t)}$ as claimed. If now $w$ is given (as in the assumptions of the theorem) we can
choose $z_i$ to be one of the points $z_r, z_l$ not equal to $w$; clearly, $z_i$ will have the desired property.
\end{proof}

\subsection{Moduli of annuli}
\label{ss:dynamic}
We need some basic notation and definitions first.

\medskip

\begin{dfn}[Renormalization and related concepts]\label{d:renorm}
A rational function $f:\P^1\to\P^1$ is \emph{renormalizable} 
if there is $q>0$ and Jordan disks $U\Subset V$ such that $f^q:U\to V$ is a degree $d^*$
polynomial-like (PL) map with \emph{connected} Julia set $K^*$ (recall that $K^*$ is the set of all points
$z$ whose $f^q$-orbits are contained in $U$).
The annulus $A=V\sm \ol{U}$ is called a \emph{fundamental annulus of $K^*$},
and the annulus $U\sm K^*$ is called a \emph{root annulus (of $K^*$)}.
The collection of sets
$$
\Kc=\{K^*,f(K^*),\dots, f^{q-1}(K^*)\}
$$
is called a \emph{PL cycle}; we always assume that $f^i(K^*)\neq f^j(K^*)$
for all $0\le i\neq j<q$. Note that each $f^i(K^*)$,
$i\ge 0$, is a connected filled Julia set of
a PL map obtained as a suitable restriction of $f^q$.
These sets, denoted $K^*,$ $\widetilde K^*$ etc, are called \emph{elements of $\Kc$} or
\emph{PL sets}.
\end{dfn}

\medskip

\begin{dfn}[Satellite PL cycles]\label{d:satel1}
Let $f$ be a rational function. Let $\Kc$ be a PL cycle of period $q$.
If there is a repelling $f$-cycle $\{\al_0$, $\dots$, $\al_{r-1}\}=\Bc$ of period $r<q$ in $\Kc^+$
(called the \emph{base cycle (of $\Kc$)}) such that $\Kc^+$ has $r$ (connected) components
$C_0,$ $\dots,$ $C_{r-1}$, and for each $i$ the set $C_i$ is the union of all elements of $\Kc$
containing $\al_i$ as a non-separating point and otherwise pairwise disjoint,
then $\Kc$ and its elements are said to be \emph{satellite} with
\emph{base period $r$} and \emph{relative period} $s$. Observe that $q=rs$ and $s>1$.
\end{dfn}

\medskip

The notation in Definitions \ref{d:renorm} and \ref{d:satel1} is used throughout.

\medskip

\begin{lem}\label{thm:1}
Let $f$ be a rational function with a satellite PL cycle $\Kc$ and $s\ge 2$.
Fix $\alpha\in \Bc$ and let $\{K_i^*=f^{r(i-1)}(K_1^*)\}_{i=1}^s$ be all elements of $\Kc$ attached to $\alpha$.
Suppose that for each $i\in \{1,...,s\}$ there are Jordan disks $U_i,$ $V_i$
such that $f^q: U_i\to V_i$ is a proper map of degree $d^*$, and a continuum
 $Z_i\subset U_i$ such that $(C_{f^q}\cap K_i^*)\cup\{\alpha\}\subset Z_i$.
There exists a critical point $w$ of $f^q$ that does not belong to $\Kc^+$.
Moreover, any point $z_j\in C_{f^q}\cap K_j^*\subset U_j$ does not belong to $U_i$
for $i\neq j$. Thus, there exists $i$ such that
$$
\mod(U_i\sm Z_i)<\frac{\pi}{\ln(4(s+1))}.
$$
\end{lem}

\begin{proof}
First we prove that in the situation of the lemma there must exist a critical point $w$ of $f^q$ that does not belong to
$\Kc^+$. Indeed, let the degree of $f$ be equal to $d\ge 2$. Consider all PL Julia set $K^*\in \Kc$ such that
$f(K^*)$ contains a critical point of $f$. If one of them is such that the degree of $f|_{K^*}$ is $d$, then this $K^*$ contains all critical points of
$f$ and its image must contain
no critical points, a contradiction (recall that $s\ge 2$). Hence for all such PL Julia sets $K^*$ the degree of $f|_{K^*}$ is less then $d$. It follows that
if $c\in f(K^*)$ is critical, then there is a preimage $d\notin \Kc^+$ of $c$; clearly, $d$ is the desired critical point $w$ of $f^q$
not belonging to $\Kc^+$.

Now, choose arbitrary points $z_j\in C_{f^q}\cap K_j^*\subset U_j$. Since the degree of each map
$f^q: U_i\to V_i$ is $d^*$ then $C_{f^q}\cap K_j^*\cap U_i=\0$ for $i\neq j$, thus $z_j\notin U_i$.
Similarly, $w\notin U_i$ for any $i$.
Set $\Pc=\{\al, z_1, \dots, z_s, w\}$. It follows that the annulus $U_i\sm Z_i$ separates $\{\al,z_i\}$ from all $z_j$ with $j\ne i$
and from $w$. Now the inequality follows from Static Theorem.
Observe that the inclusion $U_i\subset V_i$ is not assumed in Corollary \ref{thm:1}.
\end{proof}

Fix an element $K^*$ of a PL cycle $\Kc$.
Let $V\sm \ol{U}$ be a fundamental annulus of $K^*$.
Let $U_i$, $V_i$ be iterated pullbacks of $U$ and $V$ containing $K^*_i=f^i(K^*)$.
Observe that the map $f^q:U_i\to V_i$ is a PL map with PL set $K^*_i$ and
$U_i\sm K^*_i$ is a root annulus of $K^*_i$ (here $i=0$, $\dots$, $q-1$).

\medskip

\begin{lem}
  \label{l:modann1}
Let $K^*$ be a satellite PL set for a degree $d$ rational function $f$.
Then $d^*\le 2^{2d-2}$,
and for each $i=0$, $\dots$, $q-1$,
$$
\frac{\mod(U\sm K^*)}{d^*}\le \mod(U_i\sm K^*_i)\le \mod(U\sm K^*).
$$
If $f$ is a polynomial of degree $d$ then $d^*\le 2^{d-1}.$
\end{lem}

\begin{proof}
Indeed, $d^*$ is the product of $d^*_i=\deg(f^i(K^*)\to f^{i+1}(K^*))$ over $i=0$, $\dots$, $q-1$.
The function $f$ has $2d-2$ critical points, counting multiplicities.
Positive integers $d^*_i$ satisfy the inequality
$$
(d^*_0-1)+\dots+(d^*_{s-1}-1)\le 2d-2,
$$
and the maximal value of the product of $d^*_i$'s subject to this inequality
is attained when $2d-2$ numbers $d^*_i$
are equal to $2$, and all remaining $d^*_j$ are equal to $1$.
In this case, the product is $2^{2d-2}$, as claimed. The case of a polynomial is similar.
\end{proof}

\begin{proof}[Proof of the Main Theorem]
Let $f$ be a rational function of degree $d\ge 2$ with a satellite PL cycle $\Kc$ of relative period $s\ge 2$.
For any root annulus $A$ of a filled PL Julia set $K^*\in \Kc$ consider its pullback root annuli of pullbacks
of $K^*$ containing the same base point as $K^*$ itself. These root annuli together with the corresponding pullbacks of
$K^*$ satisfy the assumptions of Lemma \ref{thm:1} which implies that there is one of them of modulus which is less than or equal to
$\frac{\pi}{\ln(4(s+1))}$. Then the inequality on the modulus of $A$ claimed in the statement of the Main Theorem
follows from Lemma \ref{l:modann1}.
\end{proof}

\section{Geometry of infinitely renormalizable sets}\label{s:infrenorm}

Let us now consider geometry of infinitely-renormalizable sets.
Given a rational function $f$, denote by $\pc_f$ the closure of forward orbits of all critical values of $f$ and call it the
\emph{postcritical set (of $f$).}

\medskip

\begin{dfn}[Infinitely renormalizable sets]
\label{d:inf-renorm1}
If a sequence of PL cycles $\Kc_n$ of $f$ is such that $\Kc^+_{n+1}\subset \Kc^+_n$, and
the period of $\Kc_{n+1}$ is greater than the period of $\Kc_n$,
then the set $S=\bigcap_{n} \Kc_n^+$ is called an
\emph{infinitely renormalizable} set for $f$. Consider all critical points of $f$ that belong to $S$;
then the union of the closures of the orbits of their images is called a \emph{postcritical infinitely
renormalizable set}.
\end{dfn}

\medskip

Let us now prove a couple of topological lemmas.

\medskip

\begin{lem}\label{l:topol1}
Let $W_0$ and $W_1$ be Jordan disks and let $f:W_0\to W_1$ be a degree $k$ branched covering
which is the restriction of a rational function $\tilde f$.
Consider a Jordan curve $\Ga$ homotopic to $\partial W_1$ rel. $f(C_f)$
 and let $V$ be the Jordan disk enclosed by $\Ga$ that contains $f(C_f)$.
Then the $f$-pullback $U$ of $V$ that contains $C_f$ is a Jordan disk
 that maps onto $V$ as a branched covering map of degree $k$ with critical set $C_f\cap U$.
\end{lem}

\smallskip

Observe that since the domain of $f$ is $W_0$ then $C_f=C_{\tilde f}\cap W_0$.

\begin{proof}
One can assume that the homotopy connecting $\d W_1$ with $\Ga$
 consists of Jordan curves $\Ga_t$ parameterized by $t\in[0,1]$.
The curve $\Ga_t$ never crosses the critical values of $f$.
Therefore, the pullbacks of $\Ga_t$ also form a homotopy $\gamma_t$
 parameterized by $t\in[0,1]$ such that $\ga_0=\d W_0$ and $\gamma_1\subset f^{-1}(\Ga)$.
All $\ga_t$ enclose the same set of critical points, and all are Jordan curves.
In particular, this is true for $\ga_1$, which proves the desired.
\end{proof}

The next lemma is a consequence of well-known facts.

\medskip

\begin{lem}\label{l:topol2}
Let $U$ and $V$ be Jordan disks and let $f:U\to V$ be a degree $k$ branched covering that is the restriction of a rational function $\tilde ff$.
Let $Z\subset V$ be a full continuum such that $f(C_f)\subset Z$.
Then the full preimage of $Z$ under $f|_U$ is a full continuum containing $C_f$.
\end{lem}

\begin{proof}
Set $Z'=f^{-1}(Z)\cap U$.
The set $V\sm Z$ is a topological annulus, and $f:U\sm Z'\to V\sm Z$ is a degree $k$
 unbranched covering of this annulus.
All covering spaces of $V\sm Z$ are classified by subgroups of $\pi_1(V\sm Z)=\Z$,
 and, in particular, such a covering is determined (up to covering equivalence) by its degree.
It follows that $U\sm Z'$ is also an annulus, hence $Z'$ is a full continuum.
\end{proof}

Now we use these lemmas to study moduli of annuli whose outer boundaries are defined through the isotopy.

\medskip

\begin{lem}\label{l:pc1}
Let $f$ be a rational function with a satellite PL cycle $\Kc$
of degree $d^*$ and relative period $s\ge 2$, let $K^*\in \Kc$, and let
$V\sm \ol{U}$ be a fundamental annulus with $U\supset K^*$.
Denote the point $\Bc\cap K^*$ by $\al$.
Set $\Fc=\{\al\}\cup f^{2q}(C_{f^{2q}})$. Let $\Ga_t$ be an isotopy rel. $\Fc$ that transforms
the Jordan curve $\partial V=\Ga_0$ to a Jordan curve $\Ga_1$ that encloses a Jordan disk $W$.
 Let $Z\subset W$ be a full continuum containing $\{\al\}\cup f^{2q}(C_{f^{2q}}\cap K^*)$. Then
$$
\mod(W\sm Z)<\frac{(d^*)^2\pi}{\ln(4(s+1))}.
$$
\end{lem}

Observe that $\{\al\}\cup f^{2q}(C_{f^{2q}}\cap K^*)\subset W$ by the assumption on the isotopy.

\begin{proof}
We may assume that $r=1$ and, hence, $s=q$.
Denote by $W_i$ the $f^i$-pullback of $W$ containing $\al$ (here $1\le i\le 2q$). By Lemma \ref{l:topol1},
the sets $W_i, 1\le i\le 2q$ are Jordan disks. We claim that the critical points of $f^{2q}|_{W_{2q}}$
belong to $K^*$. Indeed, the $f^{2q}$-pullback $V_{2q}$ of $V$ containing $\al$ contains only critical points of $f^{2q}$ that belong to
$K^*$. If there exists a critical point $\be$ of $f^{2q}|_{W_{2q}}$ that does not belong to $K^*$ then $f^{2q}(\be)\in \Ga_t$
for some $t\in [0, 1]$, a contradiction with the properties of $\Ga_t$ (which is an isotopy rel. $\Fc$). So, the critical points of $f^{2q}|_{W_{2q}}$
belong to $K^*$. Moreover, since $Z$ is a full continuum containing $\{\al\}\cup f^{2q}(C_{f^{2q}}\cap K^*)$, then Lemma \ref{l:topol2}
implies that the $f^i$-pullbacks of $Z$ contained in the sets $W_i$, are full continua (here $1\le i\le 2q$), and the annuli $W_i\sm Z_i$ are $f^i$-pullbacks
of the annulus $V\sm Z$. Finally, let $K^*_i\in \Kc$ be the $f^i$-pullbacks of $K^*$
(here and in what follows $K^*_i$ should be understood with $i$ taken modulo $s=q$).

We will now consider Jordan disks $W_{q+1},$ $W_{q+2},$ $\dots,$ $W_{2q}$ that contain full continua $Z_{q+1},$ $Z_{q+2},$
$\dots,$ $Z_{2q}$, respectively. By the previous paragraph and by Lemma \ref{l:topol2}, the map $f^q:W_{q+i}\to W_i$ is a branch covering map of degree $d^*$
for any $i, 1\le i\le q$, and all points of $C_{f^q}\cap K^*_{q+i}$ belong to $Z_{q+i}$. Applying Lemma \ref{thm:1} to this
collection of Jordan disks and full continua, we find $i$, $1\le i\le q$ with
$$
\mod(W_{q+i}\sm Z_{q+i})<\frac{\pi}{\ln(4(s+1))}.
$$
On the other hand, the map $f^{q+i}|_{W_{q+i}}$ is of degree at most $(d^*)^2$. Properties of moduli of annuli then imply that
$$
\mod(W\sm Z)<\frac{(d^*)^2\pi}{\ln(4(s+1))}
$$
as desired.
\end{proof}

Notice that there are various fundamental annuli $V\sm \ol{U}$ with $U\supset K^*$; hence
there may exist several homotopy classes of $\partial V$ rel $\Fc$, and Lemma \ref{l:pc1} applies to
all annuli $W\sm Z$ associated with them.

We will need the function $\psi(x)$ defined through its inverse
$\psi^{-1}(x)=2\arcsin\left(e^{-x/2}\right)/x$ (which is decreasing from infinity at $0$ to $0$ at infinity).
Now, consider a simple closed geodesic $\gamma$ in a hyperbolic Riemann surface $X$.
By \cite{Mas85}, there always exists an annulus $A\subset X$ homotopic to $\ga$, such that
$$
\psi^{-1}(\ell(\ga))=\frac{2}{\ell(\ga)}\arcsin\left(e^{-\ell(\ga)/2}\right)\le \mod(A) \le \frac{\pi}{\ell(\ga)}.
\eqno{(\ell|m)}
$$
Here, the lower and the upper bounds have different asymptotics as $\ell(\gamma)\to\infty$.
\def\Lf{\mathfrak{L}}

\medskip

\begin{thm}
\label{t:hyper}
Let $f$ be a rational function with a satellite PL cycle $\Kc$
of degree $d^*$ and relative period $s>1$. Let $K^*\in \Kc$,
denote the point $\Bc\cap K^*$ by $\al$, and set $\Fc=\{\al\}\cup f^{2q}(C_{f^{2q}})$.
Suppose that $T\supset \Fc$ is a closed set.
Let $\widetilde X_T=\P^1\sm T$. If $\gamma$ is a Jordan curve in $\widetilde X_T$ homotopic to the outer
boundary of a fundamental annulus around $K^*$ rel. $T$
then its hyperbolic length $\ell_{\widetilde X_T}(\gamma)$ in $\widetilde X_T$ satisfies the inequality
$$
\frac {2\arcsin e^{-\frac{\ell_{\widetilde X_T}(\gamma)}{2}}}{\ell_{\widetilde X_T}(\gamma)}\le
\frac{(d^*)^2\pi}{\ln(4(s+1))}
$$
so that
$$
\ell_{\widetilde X_T}(\gamma)\ge \psi\left(\frac{(d^*)^2\pi}{\ln(4(s+1))}\right)
$$
giving $\ell_{\widetilde X_T}(\gamma)$ a lower bound of order $O(\ln\,\ln s)$ as $s\to\infty$.
\end{thm}

\begin{proof}
The theorem follows from Lemma \ref{l:pc1} and inequality $(\ell|m)$.
\end{proof}

We are ready to prove Corollary A.

\begin{proof}[Proof of Corollary A]

Consider a polynomial $f$ with a satellite PL cycle $\Kc$ of relative period $s>1$ which is a part of
an infinitely-renormalizable set $S$.
Recall that $\Bc$ is the base cycle of $\Kc$.
Suppose that all (finite) critical points of $f$ belong to $S$, and set $X=\P^1\sm (\Bc\cup \pc_f)$.
Let $\ga$ be a simple closed geodesic in $X$ homotopic in $X$ to a fundamental annulus of a PL Julia set $K^*\in \Kc^+$.
Theorem \ref{t:hyper} implies that $\ell_X(\ga)$ is bounded from below by a 
universal function of order $O(\ln\,\ln s)$ as $s\to\infty$. This proves Corollary A.
\end{proof}

\section{Theorem B and a counterexample}\label{s:b}

\begin{proof}[Proof of Theorem B]
Clearly, it is enough to prove that we have $\mod(V_n\sm K_n)\to 0$.
By way of contradiction, assume that $\mod(V_n\sm K_n)$ stay away from 0 for an infinite subsequence of numbers $n$.
Pulling $V_n$ and $U_n$ back under $f_n^{q_n}$, we may assume that
there are also Jordan disks $W_n$ with $f^{q_n}_n:V_n\to W_n$ being PL-maps.
By Lemma \ref{l:euclid}, a neighborhood of $K$ is contained in all sets $U_n$.
Passing to a subsequence, assume that there is a domain $U$ such that for all $n$ we have
$K_n\subset U\Subset U_n$, and $\mod(U\setminus K_n)\ge m$ for some $m>0$.

\begin{lem}\label{l:UinF}
The set $U$ is contained in the Fatou set of $f$.
\end{lem}

\begin{proof}
The \emph{exceptional set} of $f$ is the maximal finite subset $E_f\subset\P^1$
with the property $f^{-1}(E)\subset E$.
By \cite[Lemma 4.9]{mil06}, $E_f$ consists of one or two points.
Moreover, if $E_f$ is nonempty, then $f^2$ is Moebius conjugate to a polynomial.
Assume, by way of contradiction, that $U\cap J(f)\ne\0$.
Then, by \cite{mil06} (see Corollary 14.2 and the following remark),
the complement of $f^j(U)$ is a subset of
an arbitrarily small neighborhood of $E_f$, for all sufficiently large $j$.
Since $q_n\to\infty$, we may assume that $q_1$ is already sufficiently large, so that
$f^{q_1}(U)\cup O(E_f)=\P^1$ for a small neighborhood $O(E_f)$ of $E_f$
(if $E_f$ is empty, then $f^{q_1}(U)=\P^1$).
We may also assume that $q_1<q_2<\dots<q_n<\dots$. Consider two cases.

(1) $E_f=\0$; then $f^{q_1}(U)=\P^1$, and since $f_n^{q_1}\to f^{q_1}$ as $n\to\infty$, by compactness of $\P^1$
it follows that $f_n^{q_n}(U)=\P^1$ for large $n$, a contradiction with $f_n^{q_n}(U)\Subset V_n$.

(2) $E_f\ne\0$;
then we may assume that $f$ is a polynomial,
possibly replacing $f$ with $f^2$ and $f_n$ with $f_n^2$.
In this case $f^{q_1}(U)$ is the whole of $\C$ except, perhaps, for a small neighborhood of $\infty$
 and $V_n\supset f_n^{q_1}(U)\supset J(f_n)$ for large $n$. This is a contradiction as $f_n^{q_1}(K^*)$ is a proper subset of $J(f_n)$
 and no points of $J(f_n)$ escape $V_n$ under iterates of $f_n$.
\end{proof}

It follows that $U\subset \Omega$ where $\Omega$ is a Fatou component of $f$.
The set $\Omega$ cannot be a component of the basin of an attracting cycle of $f$
as otherwise, for large $n$, the PL set $K_n\subset U$ would be in the basin of
an attracting cycle of $f_n$, a contradiction. By way of contradiction assume that
$\Omega$ is not in the basin of a parabolic cycle of $f$.
Thus $\Omega$ is eventually mapped to a periodic rotation domain (Siegel disk or Herman ring) under $f$.
It is safe to assume that $\Omega$ is itself a periodic rotation domain of period $p$.

As $\mod(U\setminus K_n)\ge m$, there is $\delta>0$ such that $\mathrm{dist}(\partial U, K_n)>\delta$ for all $n$.
Now, since $f^p: \Omega\to\Omega$ is conjugate to an irrational rotation,
one can fix $s>0$ such that, for each $z\in \overline{U}$ we have $\mathrm{dist}(z, f^{sp}(z))<\delta/2$.
As $f_n\to f$, it follows that for every $n$ large enough,
$f_n^{sp}(K_n)\subset U\subset U_n$ and $f_n^{sp}(K_n)\neq K_n$.
The sets $K_n$ and $K'_n=f_n^{sp}(K_n)$ are in the same PL cycle and $K'_n\setminus K_n$
is nonempty set of points which do not escape $U\subset U_n$ under iterates of $f_n$, a contradiction again.
\end{proof}

All conditions (1) -- (3) are essential for the conclusion of Theorem B.
Namely, it is clear that the conclusion breaks down without condition (2).
As for conditions (1) and (3),
counterexamples can be found in the real unimodal family $f_c(z)=z^d+c$, with $c\in\R$,
for every even positive $d$. Indeed, in that case it is known that there is a universal complex bound for a specific choice
of domains for any renormalization of $f_c$ of period $s$ whenever $f_c^{2s}$
has no attracting or neutral fixed point, see Theorem A of \cite{LS}
(complemented by the following paragraph), see also \cite{GSb,LyYab}.
Thus it is enough to find corresponding examples of real $f_c$ with periodic intervals.
It follows that the conclusion breaks down without condition (1).
As for (3), the following counterexample shows that this assumption is also necessary.

\medskip

\begin{prop}
\label{ex:Ex}
One can choose a sequence $f_{a_n}(z)=z^2+a_n$ of real quadratic polynomials with connected Julia sets as follows. For each $n$ there is a renormalization
$f_{a_n}^{q_n}: U_n\to V_n$ of period $q_n$ with a single critical point at $0$, connected PL set $K_n$, such that $K_n\to K$ in the Hausdorff metric
where $K$ is non-degenerate, $q_n\to \infty$ while, for some $\delta>0$, the modulus of each fundamental annulus $V_n\setminus U_n$ is bigger than $\delta$.
\end{prop}

\begin{proof}
It is enough to choose a sequence $a_n\in [-2,0]$ such that:
\begin{enumerate}
  \item[$(i)$] the sequence $a_n$ is decreasing, and $a_n\searrow a$; 
  \item[$(ii)$] the map $f_a$ has a parabolic 3-cycle of multiplier $1$;
  \item[$(iii)$] there is a symmetric w.r.t. $0$ and $f_{a_n}$-periodic interval $L_n$ of period $q_n$ such that
   $f_{a_n}^{q_n}: L_n\to L_n$ is a unimodal map;
  \item[$(iv)$] maps $f_{a_n}^{2q_n}$ have no attracting/parabolic fixed points;
  \item[$(v)$] periods $q_n$ tend to infinity;
  \item[$(vi)$] diameters of $L_n$ stay away from $0$.
\end{enumerate}
Indeed, then, by applying the above universal complex bounds, for the sequence $f_{a_n}$ and a special choice of
renormalizations $f_{a_n}^{q_n}:U_n\to V_n$ where $L_n\subset U_n\Subset V_n$, there exists $\delta>0$ such that $\mod(V_n\sm U_n)>\delta$.

The sequence $(a_n)$ can be defined using Lavaurs maps. Here is a detailed construction
motivated by \cite{LZ}. Start with a map $f_a$ satisfying $(ii)$.
Set $F=f_a^3$, and let $x_0<0$ be the point of the parabolic $3$-cycle of $f$ that
contains $0$ in its immediate basin of attraction.
Locally near the point $x_0$, we have $F(z)=z+A(z-x_0)^2+B(z-x_0)^3+...$ where $A<0$.
Let $[c_{-2},0]$ be the maximal interval containing $x_0$ on which $F$ is increasing.
Here $c_{-2}<0$, $f_a^2(c_{-2})=0$ and $F(c_{-2})=c_1:=f_a(0)$ where $c_1<c_{-2}$.
The interval $(c_1, c_{-2})$ contains a point $c_{-1}$ such that $f_a(c_{-1})=0$ so that
$f_a^2$ is decreasing on $[c_1, c_{-1}]$ and increasing on $[c_{-1}, 0]$.
Thus we have the following order of points:
$
c_1<c_{-1}<c_{-2}<x_0.
$
Consider the fundamental interval $I_{-}=[F(0), 0]$ of the immediate attracting basin of $x_0$. This means that
the sets $F^n(I_-)$ for $n\ge 0$ have disjoint interiors, and their union covers the interval $(x_0,0]$
all of whose points converge to $x_0$ under forward iterations of $F$.
Also, consider the interval $I_{+}=[c_1, c_{-2}]$.
This is a fundamental interval for the backward iteration of $F^{-1}: [c_1, x_0]\to [c_{-2}, x_0]$
which, from now on, denote the inverse branch of strictly increasing $F: [c_{-2}, x_0]\to [c_1, x_0]$.

Now, consider attracting ($\varphi_{-}$) and repelling ($\varphi_{+}$) Fatou coordinates of $F$ near $x_0$, see, e.g., \cite{mil06,Sh}.
Recall briefly the definition.
Let $\delta>0$ be small and $D_{-}$, $D_{+}$ be the disks centered at $x_0 + \delta$, $x_0 - \delta$ respectively with radius
$\delta$. Then $F(D_{-})\subset D_{-}$ while $D_{+}\subset F(D_{+})$. This defines, after
identification of $z$ with $F(z)$ at the boundary, two cylinders $U_{-} = D\setminus F(D_{-})$
and $U_{+}=F(D_{+})\setminus D_{+}$.
By the Riemann mapping theorem, $U_{\pm}$ can be uniformized by
"straight" cylinders: there exist $\varphi_{\pm}$ mapping the cylinders $U_{\pm}$ conformally onto
vertical strips of width $1$ conjugating $F$ to $T_1$ where $T_\sigma: z\mapsto z+\sigma$ denotes the translation by $\sigma$. That is,
\begin{equation}\label{-}
\varphi_{-}(F(z))=T_1(\varphi_{-}(z))
\end{equation}
for $z\in U_{-}$ and, correspondingly,
\begin{equation}\label{+}
\varphi_{+}(F(z))=T_1(\varphi_{+}(z))
\end{equation}
for $z\in U_{+}$.
By symmetry, we may assume that $\varphi_{\pm}(\bar z)=\overline{\varphi_{\pm}(z)}$.

Then $\varphi_{-}$ extends by (\ref{-}) to an orientation reversing homeomorphism
$\varphi_{-}: (x_0, 0]\to [\varphi(0), +\infty)$
and $\varphi_{+}$ extends by (\ref{+}) to an orientation reversing homeomorphism
$\varphi_{+}: [c_1, x_0)\to (-\infty, \varphi_{+}(c_1)]$.
So, $\varphi_{-}(I_{-})=[\varphi_{-}(0), \varphi_{-}(F(0))]$
with $\varphi_{-}(F(0))=\varphi_{-}(0)+1$
and $\varphi_{+}(I_{+})=[\varphi_{+}(c_{-2}),$ $\varphi_{+}(c_1)]$
with $\varphi_{+}(c_1))=\varphi_{+}(c_{-2})+1$.
Notice that the Fatou coordinates $\varphi_{\pm}$ are unique up to post-com\-po\-sition by a real translation.
This allows us to fix the choice of $\varphi_{\pm}$ in such a way that
$X:=\varphi_{-}(0)=\varphi_{+}(c_{-2})$
which means that $\varphi_{-}((x_0, 0])=[X, +\infty)$ while
$\varphi_{+}([c_1, x_0))=(-\infty, X+1]$.
Hence, $\varphi_{-}(I_{-})=\varphi_{+}(I_{+})=[X, X+1]$ and
the following map is a well-defined orientation preserving homeomorphism:
$g_0:=\varphi_{+}^{-1}\circ \varphi_{-}: I_{-}\to I_{+}$.
More generally, let
\begin{equation}
g_\sigma=\varphi_{+}^{-1}\circ T_\sigma\circ \varphi_{-}
\end{equation}
be the Lavaurs map.  If $\sigma\le 0$ then $T_\sigma\circ \varphi_{-}(I_{-})=[X+\sigma, X+1+\sigma]\subset \varphi_{+}([c_1, x_0))$,
i.e., for each $\sigma\le 0$, the map $g_\sigma: I_{-}\to [c_1, x_0)$ is a well-defined orientation preserving homeomorphism on its image
$[g_\sigma(F(0))),$ $g_\sigma(0)]$ where $g_\sigma(F(0))=F(g_\sigma(0))$.
When $\sigma$ monotonically moves from $0$ to the left, the end points $g_\sigma(F(0))$, $g_\sigma(0)$ of the image $g_\sigma(I_{-})$ move monotonically to the right.
There is a unique $\sigma_0<0$ such that $g_{\sigma_0}(F(0))=c_{-1}$.
For every $\sigma\in [\sigma_0, 0]$, there exists a unique solution $q_\sigma\in I_{-}$ of the equation
$g_\sigma(x)=c_{-1}$, so that $g_\sigma([q_\sigma, 0])=[c_{-1}, g_\sigma(0)]$.
Note that $q_\sigma$ increases from $q_0$ to $q_{\sigma_0}$ as $\sigma$ decreases from $0$ to $\sigma_0$.

Let $G_\sigma=f^2_a\circ g_\sigma$. As $f^2_a$ increases on $[c_{-1}, 0]$, for every $\sigma\in [\sigma_0, 0]$,
$G_\sigma: [q_\sigma, 0]\to [c_1, f^2_a(g_\sigma(0))]$ is an orientation preserving homeomorphism.
The map $g_\sigma$ extends immediately by symmetry to an even map on $[F(0), -F(0)]$ (which is again denoted by $g_\sigma$). Therefore,
we get a unimodal map $G_\sigma$ on $[q_\sigma, -q_\sigma]$, for each $\sigma\in [\sigma_0, 0]$.

The following holds:
(a) $G_\sigma$ is increasing on $[q_\sigma, 0]$ and is an even function on $[q_\sigma, -q_\sigma]$, for $\sigma_0\le\sigma\le 0$,
(b) $G_\sigma(q_\sigma)=c_1<q_\sigma<0$ for $\sigma_0\le\sigma\le 0$,
(c) $G_0(0)=0$ and $G_{\sigma_0}(0)=-c_{-2}>-F(0)>-q_{\sigma_0}>0$,
(d) $G_\sigma(0)$ decreases from $G_{\sigma_0}(0)>0$ to $G_0(0)=0$ as $\sigma$ increases from $\sigma_0$ to $0$.

Indeed, (a) -- (b) hold by the construction and $G_0(0)=f^2_a(c_{-2})=0$. Now,
$G_{\sigma_0}(0)=f^2_a(F^{-1}(c_{-1}))$ is a point of $f^{-1}_a(\{c_{-1}\})$ where $c_{-1}<0$ and $f_a(c_{-1})=0$.
On the other hand, $c_1<c_{-1}<c_{-2}<F^{-1}(c_{-1})<0$ and $f^2_a$ increases on $[c_{-1}, 0]$, hence,
$0<G_{\sigma_0}(0)$. There is just one positive point of $f^{-1}_a(\{c_{-1}\})$, which is $-c_{-2}$.
To finish with (c), it remains to note that $c_{-2}<x_0<F(0)<q_{\sigma_0}$.
Finally, (d) follows from (c).

Now, (\ref{-}) allows us to extend $\varphi_{-}$ from $U_{-}$ to an analytic function to the basin of attraction $\Delta$ of the parabolic $3$-cycle of $f_a$ while the inverse map $\varphi_{+}^{-1}$ extends by (\ref{+}) from $\varphi_{+}(U_{+})$ to an entire function.
Hence, $g_\sigma$ extends to an analytic function to $\Delta$.
The main purpose of introducing $g_\sigma$ (and our use of it) is the following
theorem due to Douady and Lavaurs \cite{Do} (stated in the particular case of $3$-cycle) as follows:
{\it for every $\sigma\in\R$
there exists a sequence $a_n\searrow a$ and an increasing sequence of positive integers $N_n$ such that
$g_{\sigma}(z)=\lim_{n\to\infty} f_{a_n}^{3 N_n}(z)$ uniformly on compact subsets of $\Delta$}.

As $[F(0), -F(0)]\subset \Delta$ and $f_a$ has a negative Schwarzian derivative $Sf_a<0$ on $\R$ this theorem implies, in particular,
that $S g_\sigma\le 0$ on $[F(0), -F(0)]$. (This
also follows directly from the fact that $g_\sigma^{-1}$ extends from the real interval to a univalent function of the upper half plane into itself, see \cite{LZ}.)
 As $G_\sigma=f^2_a\circ g_\sigma$ and $Sf^2_a<0$ on $\R$, then
$SG_\sigma<0$ on $[F(0), -F(0)]$. Now, $G_\sigma$ has on $[F(0), -F(0)]$ precisely $3$ critical points:
$0$ and two symmetric critical points at $\pm q_\sigma$.
Along with relations $G_\sigma(q_\sigma)<q_\sigma<0<G_\sigma(0)$ for $\sigma_0\le \sigma<0$, $G_0(0)=0$, and general properties
of maps with the negative Schwarzian \cite{MS}
we obtain that, for each $\sigma\in [\sigma_0, 0]$ the map $G_\sigma$ has a unique (orientation preserving) fixed point $\beta_\sigma\in (q_\sigma, 0)$. Moreover, it is repelling. So we have a unimodal restriction of $G_\sigma$ to $L_\sigma:=[\beta_\sigma, -\beta_\sigma]$.
Now, as $\beta_0<0=G_0(0)<\beta_0$ while $-\beta_{\sigma_0}<G_{\sigma_0}(0)$, by continuity,
there is a Chebyshev parameter $\sigma_{Ch}\in (\sigma_0, 0)$, i.e., such that $G^2_{\sigma_{Ch}}(0)=\beta_{Ch}$.
Let us fix any $\sigma_*\in (\sigma_{Ch}, 0)$ such that $G_*:=G_{\sigma_*}: L_*\to L_*$ where $L_*=L_{\sigma_*}$
has no attracting or neutral fixed point or $2$-cycle. For example, any $\sigma_*$ close enough $\sigma_{Ch}$ would work. Notice that $L_*\subset [F(0), -F(0)]$. By the earlier stated theorem of Douady and Lavaurs,
there exists a sequence $a_n\searrow a$ and an increasing sequence of positive integers $N_n$ such that
$g_{\sigma_*}(z)=\lim_{n\to\infty} f_{a_n}^{3 N_n}(z)$ uniformly in some complex neighborhood of $[F(0), -F(0)]$.
Let $q_n=3 N_n+2$. Then, for each large enough $n$, the map $f_{a_n}$ has a symmetric periodic interval $L_n\ni 0$ of period $q_n$ and
$L_n\to L_*$ as $n\to\infty$. Besides, $f_{a_n}^{2q_n}$ has no attracting or neutral fixed point on $L_n$.
It follows that the sequence $(a_n)$ is as required.
Notice that in Example \ref{ex:Ex} 
the number of components of the orbit of $K_n$
of size at least $k$ asymptotically, as $n\to\infty$, is $C k^{-1/2}$ with some $C>0$.
\end{proof}

\end{document}